\documentclass[12pt]{amsart}
\usepackage{amssymb,latexsym,amsfonts}
\usepackage{amsmath}
\usepackage{amscd}
\usepackage{pb-diagram,pb-xy}

\usepackage{graphicx}
\usepackage{hyperref}
\usepackage{color}
\usepackage[matrix,arrow]{xy}
\newcommand{\ccc}{{ c}\,}

\newcommand{\cc}{{\rm c}}

\DeclareMathAlphabet{\mathpzc}{OT1}{pzc}{m}{it}

\newtheorem*{prop-non}{Proposition}
\newtheorem*{thm-non}{Theorem}
\newtheorem{theorem}{Theorem}[section]
\newtheorem{corollary}[theorem]{Corollary}
\newtheorem{lemma}[theorem]{Lemma}
\newtheorem{remark}{Remark}[section]
\newtheorem{proposition}[theorem]{Proposition}
\newtheorem{thm}{Theorem}[section]
\theoremstyle{definition}
\newtheorem{definition}[thm]{Definition}

\begin{document}

 \author{ Mohammed  Larbi Labbi}
 \title[Pontrjagin  classes]{ Remarks on Bianchi sums and Pontrjagin classes}
   \date{}
\thanks{}
\begin{abstract} We use the exterior and composition products of double forms together with the alternating operator to reformulate Pontrjagin classes and all Pontrjagin numbers in terms of the Riemannian curvature. We show that the alternating operator is obtained by a succession of applications of the first Bianchi sum and we prove some useful identities relating the previous four operations on double forms.\\
As an application, we prove that for a $k$-conformally flat manifold of dimension $n\geq 4k$, the Pontrjagin classes $P_i$ vanish for any $i\geq k$.\\
Finally, we study the equality case in an  inequality of Thorpe between  the Euler-Poincar\'e charateristic and the $k$-th Pontrjagin number of a $4k$-dimensional Thorpe manifold.
\end{abstract}
   \maketitle
\tableofcontents
\section{Introduction}
Let $(M,g)$ be a compact oriented  Riemannian smooth  manifold of dimension $n$, and $k$  a positive integer such that $4k\leq n$. We denote by $R$  the Riemann curvature tensor of $(M,g)$ seen as a $(2,2)$ double form.  It results from a theorem of Chern \cite{Chern} that  the following object  (which is written  using double forms operations)
\[P_k(R)=\frac{1}{(k!)^2(2\pi)^{2k}}{\rm Alt}\big(R^k\circ R^k\big),\]
is a closed differential form of degree $4k$ that represents the $k$-th pontrjagin class of $M$. In particular, if $n=4k$, the integral over $M$ of  $P_k(R)$ is a topological invariant, namely it is the $k$-th Pontrjagin number of $M$. \\
In this paper, we first study in some details the double forms operations that were used to define the above differential form, namely the ${\rm Alt}$ operator, the exterior product of double forms  that were used to define the $k$-th power of $R$ and finally the composition product $\circ$.\\
 In particular, we show that the ${\rm Alt}$ operator is obtained by a succession of applications  the first Bianchi sum $\mathfrak{S}$ in the algebra of double forms, precisely we prove that 
\begin{prop-non}
Let $\omega$ be a $(p,q)$ double form then we have
\[ {\rm Alt}(\omega)=(-1)^{pq+\frac{q(q-1)}{2}}\frac{p!}{(p+q)!}\mathfrak{S}^q\omega.\]
where $\mathfrak{S}^q=\mathfrak{S}\circ ... \circ \mathfrak{S}$, $q$-times. In particular, if $\omega$ satisfies the first Bianchi identity then ${\rm Alt}(\omega)=0.$
\end{prop-non}
We prove also another useful identity, namely
\begin{prop-non}
  Let  $q,s\geq 1$ and $p\geq 0$. If $\omega_1$ is a  $(p,q)$ double form and  $\omega_2$ is a $(q-1,s-1)$  double form  then 
\[  {\rm Alt}(g\omega_2\circ \omega_1)=(-1)^{p}\frac{p+1}{s}{\rm Alt}\Big( \omega_2   \circ  \mathfrak{S}(\omega_1)\Big).\]
Where $g$ is the Riemannian metric. In particular, if $\omega_1$ satisfies the first Bianchi identity then $ {\rm Alt}(g\omega_2\circ \omega_1)=0$.
\end{prop-non}
As a direct consequence of the previous proposition, we show that Pontrjagin classes depend only on the Weyl part of the Riemann curvature tensor $R$. In particular, all the Pontrjagin classes of a conformally-flat manifold vanish. \\
More generally, for $n\geq 4k$, following Kulkarni we shall  say that a Riemannian $n$-manifold  $(M,g)$  is  $k$-conformally flat if  its Riemann curvature tensor $R$ satisfies   $R^k=gH$. In other words, the $k$-th exterior power of $R$  is divisible by $g$ in the exterior algebra of double forms.   We recover the usual conformally-flat manifolds for $k=1$.\\
 Another  consequence of above formula is the following 
\begin{thm-non}
If  $(M,g)$ is a $k$-conformally flat Riemannian manifold of dimension  $n\geq 4k$, then the Pontrjagin classes $P_i$ of $M$ vanish for $i\geq k$.
\end{thm-non}

In section \ref{Thorpe-man}, we study  a generalization of $4$ dimensional Einstein manifolds namely Thorpe manifolds. 
An oriented compact Riemannian manifold $(M,g)$ of dimension $n=4k$ is said to be a \emph{Thorpe manifold} if $\ast R^{k}=R^{k}.$ Where $R$ is the Riemann curvature tensor seen as a $(2,2)$ double form, $R^{k}$ its exterior power and $\ast$ is the double Hodge star operator acting on double forms. The usual Einstein $4$-manifollds are obtained for $k=1$. \\
The following theorem provides topological obstructions to the existence of a Thorpe metric
\begin{thm-non}
Let $(M,g)$ be a compact orientable $4k$-dimensional Thorpe manifold, then
\[\chi(M)\geq \frac{k!k!}{(2k)!}|p_k(M).\]
Where $p_k(M)$ is the $k$-th Pontrjagin number of $M$ and $\chi(M)$  is the Euler-Poincar\'e characterestic of $M$. Furthermore if equality holds then $M$ is $k$-Ricci flat, that is $\cc R^k=0.$
\end{thm-non}
The first part of the previous theorem is originally  due to Thorpe \cite{Thorpe}, however the second part of the previous theorem is a new result. Both parts of the theorem will be proved in the last section using double forms formalism.\\
  In dimension $4$,  the second part of the previous theorem  tell us that a compact orientable Einstein manifold $M$ that  satisfies $2\chi(M)=|p_1(M)|$  must be Ricci-flat. We note that Hitchin \cite{Hitchin} showed that such a metric must be then 
 either flat or a Ricci flat Kahler metric on the $K3$ surface (Calabi-Yau metric), or a quotient of it.\\
In dimension $8$, Kim \cite{Kim2} showed  if equality holds in the previous theorem (that is $6\chi(M)=|p_2(M)|$) and if we assume also that the metric is Einstein then the metric must be flat. We prove the following generalization of Kim's result
\begin{thm-non}
A compact orientable manifold $M$  of dimension $8k$ that is at the same time hyper $2k$-Einstein and Thorpe and satisfies $\frac{(4k)!}{(2k)!(2k)!}\chi(M)=|p_{2k}(M)|$  must be $k$-flat, that is $R^k=0$.
\end{thm-non}
Where a metric is said to be hyper $2k$-Einstein if  its Riemann curvature tensor satisfies $\cc R^k =\lambda g^{2k-1},$ for some constant real number $\lambda$. We recover the usual Einstein metrics for $k=1$.\\
In the last section, we prove a Stehney type formula for all the (mixed) Pontrjagin numbers  $p_1^{k_1}p_2^{k_2}\cdots p_m^{k_m}$. Recall that these numbers  are important topological  invariants for a manifold, they
are oriented cobordism invariant, and together with Stiefel-Whitney numbers they determine an oriented manifold's oriented cobordism class. Furthermore, the signature and the $\hat{A}$ genus can be expressed explicitely through linear combinations of   Pontrjagin numbers.
\begin{thm-non}
Let $(M,g)$ be a compact $4k$-dimensional Riemannian oriented manifold  with Riemann curvature tensor $R$, seen as a $(2,2)$ double form,  and let $ k_1, k_2, ..., k_m$  be a  collection of natural numbers  such that $k_1+2k_2+...+mk_m =k$.Then the Pontrjagin number $p_1^{k_1}p_2^{k_2}\cdots p_m^{k_m}$  of $M$  is given by the integral over $M$ of the following $4k$-form
\begin{equation*}
 P_1^{k_1}P_2^{k_2}\cdots P_m^{k_m}=\frac{(4k)!}{[(2k)!]^2(2\pi)^{2k}}  \Big(\prod_{i=1}^{m}\frac{[(2i)!]^2}{(i!)^{2k_i}(4i)!}\Big){\rm Alt}\Big[(R\circ R)^{k_1}(R^2\circ R^2)^{k_2}\cdots (R^m\circ R^m)^{k_m}\Big].
\end{equation*}
Where all the powers over double forms are taken with respect to the exterior product of double forms.
\end{thm-non}
\section{ The Exterior and composition  Algebras  of Double Forms}\label{prem.}
\subsection{The exterior algebra of double forms}

Let $(V,g)$ be an Euclidean real vector space  of finite dimension $n$.  Let
  $\Lambda V^{*}=\bigoplus_{p\geq 0}\Lambda^{p}V^{*}$  (resp.
   $\Lambda V=\bigoplus_{p\geq 0}\Lambda^{p}V$) denotes the exterior algebra
 of the dual space  $V^* $ (resp.   $V$). We define \emph{the space of exterior  double forms} of $V$  (resp. \emph{the space of exterior double vectors}) as
 $${\mathcal D}(V^*)= \Lambda V^{*}\otimes \Lambda V^{*}=\bigoplus_{p,q\geq 0}
  {\mathcal D}^{p,q}(V^*),$$
(resp.
 $${\mathcal D}(V)= \Lambda V\otimes \Lambda V=\bigoplus_{p,q\geq 0}
  {\mathcal D}^{p,q}(V)),$$
 where $  {\mathcal D}^{p,q}(V^*)= \Lambda^{p}V^{*} \otimes  \Lambda^{q}V^{*},$ (resp.  $  {\mathcal D}^{p,q}(V)= \Lambda^{p}V \otimes  \Lambda^{q}V).$\\
The space ${\mathcal D}(V^*)$ is naturally  a bi-graded associative  algebra, called \emph{double exterior algebra of $V^*$},
 where for simple elements  $\omega_1=\theta_1\otimes \theta_2\in { \mathcal D}^{p,q}(V^*)$ and
 $\omega_2=\theta_3\otimes \theta_4\in  {\mathcal D}^{r,s}(V^*)$, the multiplication is given  by
 \begin{equation}
 \label{def:prod}
 \omega_1\omega_2= (\theta_1\otimes \theta_2 )(\theta_3\otimes
 \theta_4)=
 (\theta_1\wedge \theta_3 )\otimes(\theta_2\wedge \theta_4)\in
    {\mathcal D}^{p+r,q+s}(V^*).\end{equation}
Where $\wedge$ denotes the standard exterior product on the exterior algebra $ \Lambda V^{*}$.\\
A \emph{double form of degree $(p,q)$} is by definition an  element of the tensor product
    $  {\mathcal D}^{p,q}(V^*)= \Lambda^{p}V^* \otimes \Lambda^{q}V^*$.\\
 The above multiplication in ${\mathcal D}(V^*)$  is called the \emph{exterior product of double forms.}
\subsection{The composition algebra of double forms}
We define a second product in the space of double exterior vectors ${\mathcal D(V)}$  (resp. in the space of double exterior forms ${\mathcal D(V^*)}$ ) which will be denoted by $\circ$ and will be called the \emph{composition product}, see \cite{Greub-book, Labbialgebraic}. Given 
 $\omega_1=\theta_1\otimes \theta_2\in { \mathcal D}^{p,q}$ and
    $\omega_2=\theta_3\otimes \theta_4\in  {\mathcal D}^{r,s}$ two simple double exterior  forms ( or double exterior vectors), we  set
    \begin{equation}
\omega_1\circ\omega_2=(\theta_1\otimes \theta_2)\circ (\theta_3\otimes \theta_4)=\langle \theta_1,\theta_4\rangle \theta_3\otimes \theta_2\in  {\mathcal D}^{r,q}.
\end{equation}
It is clear that $\omega_1\circ\omega_2=0$ unless $p=s$. Then one can extends the definition to all double forms using linearity. We emphasize that, in contrast with the exterior product of double forms, the composition product clearly depends on the metric $g$.\\
It turns out that the space of double forms endowed with the composition product  $\circ$ is  an associative algebra.\\
 For a double form (or vector) $\omega \in  {\mathcal D}^{p,q}$, we denote by $\omega^t\in  {\mathcal D}^{q,p}$ the transpose of $\omega$. For a simple double form it is defined by
\begin{equation}
\big( \theta_1\otimes \theta_2\big)^t=\theta_2\otimes \theta_1.
\end{equation} 
Using linearity the previous definition can be extended to all double forms. Note that a double form  $\omega$ is a symmetric double form if and only if $\omega^t=\omega$. 

\section{Basic maps in the space  of double forms}
Let $(e_1,...,e_n)$ be an orthonormal basis of $V$ and $(e_1^*,...,e_n^*)$ its dual basis of $1$-forms.\\
Let  $h=e_i^*\otimes e_j^*\in \mathcal{D}^{(1,1)}(V^*)$ be a simple double form,  we define four basic  linear  maps $\mathcal{D}(V^*)\rightarrow \mathcal{D}(V^*)$ as follows. For $\omega=\theta_1\otimes \theta_2 \in \mathcal{D}(V^*)$ we set
\begin{align*} 
L^{1,1}_h(\omega)&=e_i^*\wedge \theta_1 \otimes e^*_j\wedge \theta_2,\\
L^{-1,-1}_h(\omega)&=i_{e_i} \theta_1 \otimes i_{e_j} \theta_2,\\
L^{1,-1}_h(\omega)&=e_i^*\wedge \theta_1 \otimes i_{e_j}\theta_2,\\
L^{-1,1}_h(\omega)&=i_{e_i} \theta_1 \otimes e^*_j\wedge \theta_2.
\end{align*}
Where $i_{e_j}$ denotes the usual interior product. We use then linearity to extend first the previous maps  $L_h^{a,b}$  to  $\mathcal{D}(V^*)$ and secondly to define them for any $h=\in \mathcal{D}^{(1,1)}(V^*)$. For instance, for the inner product  $g=\sum_{i=1}^n e_i^*\otimes e_i^*\in \mathcal{D}^{(1,1)}(V^*)$ we have

$$L_g^{a,b}(\omega)=\sum_{i=1}^n L_{e_i^*\otimes e_j^*}^{a,b}(\omega).$$
Recall that the exterior algebra $\Lambda V^*$  inherits naturally anl inner product, denoted $\langle .,.\rangle$,  from the inner product $g$ of $V$. This inner product can be canonically extended to an inner product   denoted also by $\langle .,.  \rangle$ on the space of double forms
 $\mathcal{D}(V^*)$ . Precisely, for simple double forms in $\mathcal{D}^{(p,q)}(V^*) $ we set
$$\langle \theta_1\otimes \theta_2, \theta_3\otimes \theta_4\rangle=\langle \theta_1,\theta_3\rangle\langle \theta_2,\theta_4\rangle.$$

With respect to the previous canonical inner product we have
\begin{proposition} For any $h\in \mathcal{D}^{(1,1)}(V^*)$ we have
\begin{enumerate}
\item The map $L^{-1,-1}_h$ is the adjoint map of   $L^{1,1}_h$.
\item  The map $L^{-1,1}_h$ is the adjoint map of   $L^{1,-1}_h$.
\end{enumerate}
\end{proposition}
\begin{proof}
Straightforward, it results directly  from the fact that the interior product by a vector $v$ is the adjoint of the exterior multiplication  by $v^*$.
\end{proof}
The following proposition shows that the previous maps are the basic maps of the algebra of curvature structures as defined in \cite{Kulkarni}
\begin{proposition}\label{Bianchi-sum}
Let   $(V,g)$ be an Euclidean vector space and  $\omega \in \mathcal{D}(V^*)$ an arbitrary double form then  we have
\begin{enumerate}
\item $L^{1,1}_g(\omega)=g\omega$ is the left multiplication map by $g$.
\item   $L^{-1,-1}_g(\omega)=\ccc \omega$ is the contraction map of double forms.
\item  $L^{1,-1}_g(\omega)=\mathfrak{S}\omega$ is the first Bianchi sum.
\item  $L^{-1,1}_g(\omega)=\widetilde{\mathfrak{S}}\omega$ is the adjoint  first Bianchi sum.
\end{enumerate}
\end{proposition}
\begin{proof}
The proof of $(1)$ is straightforward, $(2)$ reults from the fact that the contraction map is the adjoint of the multiplication map by $g$, see \cite{Labbidoubleforms}. To prove (3) we proceed as follows\\
Let $\omega=e_{i_1}^*\wedge ... \wedge e_{i_p}^*\otimes e_{j_1}^*\wedge ... \wedge e_{j_q}^*$ be a simple $(p,q)$ double form then
\begin{align*}
L^{1,-1}_g(\omega)=&\sum_{i=1}^n  e_i^*\wedge e_{i_1}^*\wedge ... \wedge e_{i_p}^*\otimes i_{e_i}e_{j_1}^*\wedge ... \wedge e_{j_q}^*\\
=&\sum_{i=1}^n  e_i^*\wedge e_{i_1}^*\wedge ... \wedge e_{i_p}^*\otimes  \sum_{k=1}^q (-1)^{q-1}\langle e_i^*,e_{j_k}^*\rangle e_{j_1}^*\wedge ... \widehat{e_{j_k}}  ... \wedge e_{j_p}^*\\
=&\sum_{i=1}^n   \sum_{k=1}^q (-1)^{q-1}\langle e_i^*,e_{j_k}^*\rangle  e_i^*\wedge e_{i_1}^*\wedge ... \wedge e_{i_p}^*\otimes  e_{j_1}^*\wedge ... \widehat{e_{j_k}}  ... \wedge e_{j_p}^*\\
=& \sum_{k=1}^q (-1)^{q-1}    e_{j_k}^*\wedge e_{i_1}^*\wedge ... \wedge e_{i_p}^*\otimes  e_{j_1}^*\wedge ... \widehat{e_{j_k}}  ... \wedge e_{j_p}^*\\
=&\mathfrak{S}\omega.
\end{align*}
In the same way one can prove that
\[ L^{-1,1}_g(\omega)=\sum_{k=1}^q (-1)^{q-1}    e_{i_1}^*\wedge ...\widehat{e_{i_k}} ... \wedge e_{i_p}^*\otimes   e_{i_k}^*\wedge e_{j_1}^*\wedge   ... \wedge e_{j_p}^*=\widetilde{\mathfrak{S}}\omega.\]
This completes the proof of the proposition.
\end{proof}
\begin{remark}
We remark  that the first Bianchi sum and its adjoint are related by the formula (which can be easily checked)
\[ \widetilde{\mathfrak{S}}(\omega)=\Big(  \mathfrak{S} (\omega^t)\Big)^t.\]
\end{remark}
As a first direct application of the previous prposition  we show that the Bianchi sum is an antiderivation as was observed in \cite{Kulkarni}
\begin{proposition}
Let  $\omega_1\in \mathcal{D}^{(p,q)}(V^*)$ and $\omega_2=\in \mathcal{D}^{(r,s)}(V^*)$ then
\begin{equation}\label{Bianchi-derivation}
\mathfrak{S}(\omega_1\omega_2)=(\mathfrak{S} \omega_1)  \omega_2+(-1)^{p+q}\omega_1\mathfrak{S}\omega_2.
\end{equation}
\end{proposition}

\begin{proof}
Without loss of gennerality we may assume that $\omega_1=\theta_1\otimes \theta_2\in \mathcal{D}^{(p,q)}(V^*)$ and $\omega_2=\theta_3\otimes \theta_4\in \mathcal{D}^{(r,s)}(V^*)$ . Then
\begin{align*}
\mathfrak{S}(\omega_1\omega_2)=& \mathfrak{S}(\theta_1\wedge \theta_3\otimes  \theta_2\wedge \theta_4)\\
=& \sum_{i=1}^n  e_i\wedge \theta_1\wedge \theta_3\otimes  i_{e_i}( \theta_2\wedge \theta_4)\\
=& \sum_{i=1}^n  e_i\wedge \theta_1\wedge \theta_3\otimes  \Big(  i_{e_i}( \theta_2)\wedge \theta_4) +(-1)^q \theta_2\wedge i_{e_i}(\theta_4)\Big)\\
=&  \sum_{i=1}^n \Big( e_i\wedge \theta_1\otimes i_{e_i}\theta_2\Big)(\theta_3\otimes \theta_4)+ \sum_{i=1}^n (-1)^p \theta_1\wedge e_i\wedge \theta_3\otimes (-1)^q\theta_2\wedge i_{e_i}(\theta_4)\\
=&  \sum_{i=1}^n \Big( e_i\wedge \theta_1\otimes i_{e_i}\theta_2\Big)(\theta_3\otimes \theta_4)+ \Big(\theta_1\otimes \theta_2\Big)(-1)^{p+q}\sum_{i=1}^n e_i\wedge \theta_3\otimes  i_{e_i}(\theta_4)\\
=&(\mathfrak{S} \omega_1)  \omega_2+(-1)^{p+q}\omega_1\mathfrak{S}\omega_2.
\end{align*}
\end{proof}
Recall that  the alternating operator is a  linear map $ \mathcal{D}(V^{*})\rightarrow  \Lambda(V^*)$  defined as follows
\begin{equation}\label{Alt-conv}
 {\rm Alt}(\theta_1\otimes \theta_2)=\frac{p!q!}{(p+q)!}\theta_1\wedge \theta_2,
\end{equation}
where $\theta_1\otimes \theta_2\in    {\mathcal D}^{p,q}(V^*)$.
It is clear that  ${\rm Alt}$ is a surjective map. However, it is far from being injective. The next proposition shows in particular that the kernel of  ${\rm Alt}$ contains all doublle forms satisfying the first Bianchi identity.
\begin{proposition}
Let $\omega$ be a $(p,q)$ double form then we have
\[ {\rm Alt}\omega=(-1)^{pq+\frac{q(q-1)}{2}}\frac{p!}{(p+q)!}\mathfrak{S}^q\omega.\]
where $\mathfrak{S}^q=\mathfrak{S}\circ ... \circ \mathfrak{S}$, $q$-times.
\end{proposition}
\begin{proof}
By linearity we may assume that $\omega=\theta_1\otimes \theta_2$ is a decomposed $(p,q)$ double form. Using proposition \ref{Bianchi-sum} we get
\begin{align*}
\mathfrak{S}^p\omega &=\Big( L_g^{1,-1}\Big)^p(\omega)\\
&= \sum_{i_1,...,i_q=1}^n  e_{i_1}^*\wedge ... \wedge e_{i_q}^*\wedge \theta_1\otimes i_{e_{i_q}\wedge ...\wedge e_{i_1}}\theta_2\\
&= \sum_{i_1,...,i_q=1}^n  e_{i_1}^*\wedge ... \wedge e_{i_q}^*\wedge \theta_1\otimes    \langle e_{i_q}^*\wedge ...\wedge e_{i_1}^* , \theta_2\rangle\\
&= \sum_{i_1,...,i_q=1}^n   \langle e_{i_q}^*\wedge ...\wedge e_{i_1}^* , \theta_2\rangle e_{i_1}^*\wedge ... \wedge e_{i_q}^*\wedge \theta_1\\
&= (-1)^{\frac{q(q-1)}{2}}  \sum_{i_1,...,i_q=1}^n   \langle e_{i_1}^*\wedge ...\wedge e_{i_q}^* , \theta_2\rangle e_{i_1}^*\wedge ... \wedge e_{i_q}^*\wedge \theta_1\\
&=q!  (-1)^{\frac{q(q-1)}{2}}\theta_2\wedge \theta_1=q!  (-1)^{pq+\frac{q(q-1)}{2}}\theta_1\wedge \theta_2= (-1)^{pq+\frac{q(q-1)}{2}}      \frac{(p+q)!}{p!} {\rm Alt}(\omega).
\end{align*}
\end{proof}

The next proposition shows that the operator ${\rm Alt}$ defines an exterior algebra endomorphism

\begin{proposition}\label{Alt-is-endom}
Let  $\omega_1\in \mathcal{D}^{(p,q)}(V^*)$,  $\omega_2\in \mathcal{D}^{(r,s)}(V^*)$ then $\omega_1\omega_2\in  \mathcal{D}^{(p+r,q+s)}(V^*)$  and
\begin{equation}\label{Alt-endomorphism}
\frac{(p+r+q+s)!}{(p+r)!(q+s)!}{\rm Alt} (\omega_1\omega_2)=(-1)^{qr}\Big(\frac{(p+q)!}{p!q!}{\rm Alt} \omega_1\Big)\wedge  \Big(\frac{(r+s)!}{r!s!}{\rm Alt} \omega_2\Big).
\end{equation}
In particular, if  ${\rm Alt}(\omega_1)=0$ for some double form $\omega_1$, then ${\rm Alt}(\omega_1\omega_2)=0$ for any double form $\omega_2$.
\end{proposition}

\begin{proof}
Without loss of gennerality we may assume that $\omega_1=\theta_1\otimes \theta_2\in \mathcal{D}^{(p,q)}(V^*)$ and $\omega_2=\theta_3\otimes \theta_4\in \mathcal{D}^{(r,s)}(V^*)$ . Then
\begin{align*}
{\rm Alt} (\omega_1\omega_2)=&{\rm Alt} (\theta_1\wedge \theta_3\otimes  \theta_2\wedge \theta_4)\\
=&\frac{(p+r)!(q+s)!}{(p+r+q+s)!}\theta_1\wedge \theta_3\wedge \theta_2\wedge \theta_4\\
=& \frac{(p+r)!(q+s)!}{(p+r+q+s)!}(-1)^{qr}\theta_1\wedge \theta_2\wedge \theta_3\wedge \theta_4\\
=& (-1)^{qr} \frac{(p+r)!(q+s)!}{(p+r+q+s)!}\frac{(p+q)!(r+s)!}{p!q!r!s!}  {\rm Alt} \omega_1\ \wedge  {\rm Alt} \omega_2.
\end{align*}
\end{proof}

We shall say that a double form $\omega$  satisfies the first Bianchi identity if   $\mathfrak{S}\omega=0$.  In the next proposition we use Bianchi first sum to  reformulate  the classical Pl\"ucker relations, called also the Grassman quadratic $p$-relations
\begin{proposition}
A $p$-form $\alpha\in \Lambda V^*$ is decomposable if and only if the double form $\alpha\otimes \alpha  \in   {\mathcal D}^{p,p}(V^*)$ satisfies the first Bianchi identity.
\end{proposition}

\begin{proof}
Let $(e_1,...,e_n)$ be an orthonormal basis of $V$. We look at double forms as multilinear forms on $V$ and we use an alternative form of the first Bianchi sum as follows
\begin{align*}
\mathfrak{S}(\alpha\otimes \alpha)& \big(e_{i_1},...,e_{i_{p+1}};e_{j_1},...,e_{j_{p-1}}\big)\\
&=\sum_{k=1}^{p+1}(-1)^k (\alpha\otimes \alpha) \big(e_{i_1}, ... \widehat{e_{i_k}},...,e_{i_{p+1}};e_{i_k},e_{j_1},...,e_{j_{p-1}}\big)\\
&=\sum_{k=1}^{p+1}(-1)^k \alpha\big(e_{i_1}, ... \widehat{e_{i_k}},...,e_{i_{p+1}}\big) \alpha \big(e_{i_k},e_{j_1},...,e_{j_{p-1}}\big).
\end{align*}
\end{proof}
The following lemma will be useful in proving some results of the next section.
\begin{lemma}\label{Altgomega}
\begin{enumerate}
\item
Let  $r,p\geq 1$, $q\geq 0$  and $\omega_1\in   {\mathcal D}^{p,q}(V^*)$, $\omega_2\in   {\mathcal D}^{r-1,p-1}(V^*)$  be two double forms then 
\begin{equation*}
 {\rm Alt}(\omega_1\circ g\omega_2)=(-1)^{r-1}\frac{q+1}{r}{\rm Alt}\Big( \widetilde{\mathfrak{S}}(\omega_1)\circ \omega_2\Big).
\end{equation*}
In particular, if $\omega_1^t$ satisfies the first Bianchi identity then $ {\rm Alt}(\omega_1\circ g\omega_2)=0$.
\item   Let  $q,s\geq 1$, $p\geq 0$  and $\omega_1\in   {\mathcal D}^{p,q}(V^*)$, $\omega_2\in   {\mathcal D}^{q-1,s-1}(V^*)$  be two double forms then 
\begin{equation*}
 {\rm Alt}(g\omega_2\circ \omega_1)=(-1)^{p}\frac{p+1}{s}{\rm Alt}\Big( \omega_2   \circ  \mathfrak{S}(\omega_1)\Big).
\end{equation*}
In particular, if $\omega_1$ satisfies the first Bianchi identity then $ {\rm Alt}(g\omega_2\circ \omega_1)=0$.
\end{enumerate}
\end{lemma}

\begin{proof}
Let $\omega_1=\theta_1\otimes \theta_2\in   {\mathcal D}^{p,q}(V^*)$ and  $\omega_2=\theta_3\otimes \theta_4\in   {\mathcal D}^{r-1,p-1}(V^*)$.\\
From one side, using the definition of the composition product of double forms  we have,
\[ \omega_1\circ g\omega_2=\big(\theta_1\otimes \theta_2\big) \circ  \sum_{i=1}^ne_i^*\wedge \theta_3\otimes e_i^*\wedge \theta_4= \sum_{i=1}^n\langle \theta_1,e_i^*\wedge \theta_4\rangle e_i^*\wedge \theta_3\otimes  \theta_2.\]
On the other hand, recall that the first adjoint Bianchi sum is given by
\[\widetilde{\mathfrak{S}}(\omega_1)=\sum_{i=1}^n  i_{e_i}\theta_1\otimes e_i^*\wedge \theta_2,\]
therefore,
\[ \widetilde{\mathfrak{S}}(\omega_1)\circ \omega_2= \sum_{i=1}^n\langle \theta_1,e_i^*\wedge \theta_4\rangle     \theta_3\otimes  e_i^*\wedge \theta_2.\]
To complete the proof it is enough to check that 
\[ {\rm Alt}(\theta_3\otimes e_i^*\wedge\theta_2)= (-1)^{r-1}\frac{q+1}{r}{\rm Alt}(e_i^*\wedge \theta_3\otimes \theta_2 ).\]
The proof of the second statement is identical.
\end{proof}

\section{Pontrjagin  forms and scalars}
\subsection{Pontrjagin forms}\label{Pontrjagin-forms}
\begin{definition}
Let $R$ be a $(2,2)$ double form on the Euclidean vector space $(V,g)$. The \emph{$k$-th Pontrjagin form} of $R$,  denoted $P_k(R)$, is the $4k$-form defined by

\begin{equation}\label{Stehney-formula}
P_k(R)=\frac{1}{(k!)^2(2\pi)^{2k}}{\rm Alt}\big(R^k\circ R^k\big)\in \Lambda^{4k}(V^*).
\end{equation}

\end{definition}

\begin{remark}
The previous definition is motivated by Riemannian geometry. If $(M,g)$ is a Riemannian manifold of dimension $n$, and $k$ is a positive integer such that $4k\leq n$. Let $R$ be the Riemann curvature tensor seen as a $(2,2)$ double form. In Stehney \cite{Stehney} it is showed that the differential form $P_k(R)$
is a closed differential form of degree $4k$ that represents the $k$-th pontrjagin class of $M$.  Let us note that the previous result of Stehney is a reformulation of a result originally due to  Chern \cite{Chern}.
\end{remark}

\begin{proposition}
Let $W$ denote the Weyl part of the $(2,2)$ double form $R$, then
\begin{equation}
P_k(R)=P_k(W)=\frac{1}{(k!)^2(2\pi)^{2k}}{\rm Alt}\big(W^k\circ W^k\big).
\end{equation}
In other words, the Pontrjagin forms depend only on the Weyl part of $R$.
\end{proposition}
This result were first proved by Avez \cite{Avez}, Bivens \cite{Bivens}, Greub \cite{Greub-pontjagin} and Branson-Gover \cite{Branson-Gover}.

\begin{proof}
Recall  that we have the decomposition $R=W+gh$, where $h$ is the Schouten tensor. Then $R^k=W^k+gH$, for some double form $H$, and therefore
\[ R^k\circ R^k=(W^k+gH)\circ (W^k+gH)=W^k\circ W^k+W^k\circ gH+gH\circ W^k+gH\circ gH.\]
Since $W^k$ and $gH$ are both symmetric double forms and  both satisfy the first Bianchi identity, Lemma \ref{Altgomega} shows that
\[   {\rm Alt}\big(W^k\circ gH\big)= {\rm Alt}\big(gH\circ W^k\big)={\rm Alt}\big(gH\circ gH\big)=0.\]
The proposition follows then immediately.
\end{proof}
As a direct consequence of the previous proposition all the Pontrjagin classes of a conformally-flat manifold vanish. In the next theorem we are going to generalize this to more general  $q$-conformally-flat manifolds.\\
For $n\geq 4k$, following \cite{Kulkarni,Nasu},   We say that a Riemannian $n$-manifold  $(M,g)$  is  $k$-conformally flat if  its Riemann curvature tensor $R$ is such that  the double form   $R^k$ is divisible by $g$, that is $R^k=gH$ for some double form $H$ satisfying the first Bianch identity. We recover the usual conformally-flat manifolds for $k=1$.  Using  Lemma \ref{Altgomega}  one can prove easily the following
\begin{theorem}
Let  $(M,g)$ be a $k$-conformally flat Riemannian manifold of dimension  $n\geq 4k$. Then the Pontrjagin classes $P_i$ of $M$ vanish for $i\geq k$.
\end{theorem}
\begin{proof}
Let  $R^k=gH$, where $H$ is some symmetric double form that satisfies the first Bianchi identity, then Lemma \ref{Altgomega} shows that
\[{\rm Alt}(R^k\circ R^k)={\rm Alt}(gH\circ R^k)=0.\]
This completes the proof.
\end{proof}

\subsection{Pontrjagin-Chern scalars of double forms}
\subsubsection{The volume double form}
 
Let $(V,g)$ be an  Euclidean vector space of even dimension $n=4k$. Fix a volume form $\omega_g$ on $V$, that is choose $\omega_g\in \Lambda^n(V^*)$ with norm 1. In other words we choose an orientation on $V$. Using the inclusion
 $\Lambda^{4k}(V^*)\subset {\mathcal D}^{(2k,2k)}(V^*)$, one can look at $\omega_g$ as a $(2k,2k)$ double form. The so obtained double form, which we continue to denote by $\omega_g$, shall be called the  \emph{volume double form of $(V,g)$.}\\
It is easy to check that the operator $ \Lambda^{2k}(V^*)\rightarrow  \Lambda^{2k}(V^*)$  that is canonically associated to $\omega_g$ is nothing but the Hodge star operator $*$. In particular, for any two $2k$-vectors $\alpha,\beta$ one has
\[\omega_g(\alpha,\beta)=\langle \ast\alpha,\beta\rangle=\langle \alpha,\ast\beta\rangle=\ast (\alpha \wedge \beta).\]

We list  below some useful properties of this volume double form.
\begin{proposition}\label{mu-star}
\begin{enumerate}
\item In the composition algebra of double forms  ${\mathcal D}(V^\ast)$, the volume double form $\omega_g$ is a squre root of the unity, that is
\[\omega_g\circ\omega_g=\frac{g^{2k}}{(2k)!}.\]
\item For any two $(2k,2k)$ double forms $\omega_1$ and $\omega_2$ we have
\[\langle \omega_g\circ \omega_1,\omega_g\circ \omega_2\rangle=\langle \omega_1,\omega_2\rangle =\langle \omega_1\circ  \omega_g,\omega_2\circ \omega_g \rangle. \]
\item If $\ast$ denotes the double Hodge star operating on double forms and $\psi$ is an arbitrary  $(2k,2k)$ double form then
\[\big(\ast\psi\big)\circ \omega_g=\omega_g\circ \psi, \, {\rm and}\,\,  \psi \circ \omega_g=\omega_g\circ  \big(\ast\psi\big).\]
In particular, we have

\[\ast\psi=\omega_g\circ \psi \circ \omega_g.\]

\end{enumerate}
\end{proposition}
\begin{proof}
Let $(e_i)$ be a positive orthonormal basis of $V$, then it is easy to show that
\[\omega_g=\sum_Ie_I^*\otimes \ast e_I^*,\]
where the index $I$ runs over all multi-indexes $i_1,...,i_{2k}$ such that $1\leq i_1<...<i_{2k}\leq n$ and the star power denotes duals as usual. Consequently,
\[\omega_g\circ\omega_g=\sum_{I,J}\langle e_I^*,\ast e_J^*\rangle e_J^*\otimes \ast e_I^*=\sum_{I}  e_I^*\otimes e_I^*=\frac{g^{2k}}{(2k)!}.\]
This completes the proof of (1).\\
Next note  that
\[  \big( \omega_g\circ \omega_1\big)^t\circ (\omega_g\circ \omega_2)=\omega_1^t\circ \omega_g\circ \omega_g\circ \omega_2=\omega_1^t\circ\omega_2.\]
Where we used the fact that $\omega_g$ is a symmetric double form. After taking full contractions of both sides  one gets the  relation in (2).\\
To prove (3), without loss of generality  let $\omega=e_I^*\otimes e_J^*$ where $I$ and $J$ are multi-indexes as above. Then $\ast \omega=\ast e_I^*\otimes \ast  e_J^*$  and
\[\big(\ast\omega\big)\circ \omega_g =\sum_K\langle \ast e_I^*,\ast  e_K ^*\rangle e_K^* \otimes \ast e_J^*  =   e_I^* \otimes \ast e_J^*.\]
On the other hand, we have
\[\omega_g\circ \omega=\sum_K \langle e_K^*, e_J^*\rangle  e_I^*\otimes \ast e_K^*=e_I^* \otimes \ast e_J^*.\]
This completes the proof of the proposition.
\end{proof}
The proof of the following properties is straightforward
\begin{proposition} If $\omega_g$ is the $(2k,2k)$ volume double form then we have
\[\ast\omega_g=\omega_g,\, \ccc\omega_g=0,\, g\omega_g=0,\, {\rm and}\, \langle \omega_g, \omega_g \rangle=1.\]
\end{proposition}

\subsubsection{Pontrjagin-Chern scalars of double forms}
\begin{definition}
Let $R$ be a $(2,2)$ double form on an oriented  Euclidean vector space $(V,g)$ of dimension $n=4k$. The \emph{$k$-th Pontrjagin-Chern scalar} of $R$,  denoted ${p}_k(R)$, is the real number defined by

\begin{equation}
{p}_k(R)=\frac{1}{(k!)^2(2\pi)^{2k}} \langle R^k\circ R^k,\omega_g\rangle.                       
\end{equation}

Where $\omega_g$ is the volume double form of $V$.

\end{definition}

It is easy to show that
\[{p}_k(R)=\frac{1}{(k!)^2(2\pi)^{2k}} \langle{\rm Alt }\big( R^k\circ R^k\big),\omega_g\rangle=\langle P_k(R),\omega_g\rangle,           \]
where the previous inner product is the one of $4k$-forms, $\omega_g$ is the volume form of $(V,g)$ and $P_k(R)$ is the $k$-th Pontrjagin form of $R$. In particular, since $\dim \Lambda^{4k}V^*=1$  we have
\[P_k(R)=\langle P_k(R),\omega_g\rangle \omega_g=p_k(R)\omega_g.\]
The following proposition provides an alternative definition of the previous scalars
\begin{proposition}\label{prop2.5}
 The \emph{$k$-th Pontrjagin-Chern scalar} ${p}_k(R)$ is given by a full contraction as follows
\begin{equation}
{p}_k(R)=\frac{1}{(k!)^2(2\pi)^{2k}(2k)!} \cc^{2k}\big(  R^k\circ R^k\circ \omega_g\big).    
\end{equation}
\end{proposition}

\begin{proof}
It results from proposition \ref{mu-star} that
\begin{align*} {p}_k(R)=& \frac{1}{(k!)^2(2\pi)^{2k}} \langle R^k\circ R^k,\omega_g\rangle=\frac{1}{(k!)^2(2\pi)^{2k}} \langle R^k\circ R^k\circ \omega_g ,\omega_g\circ \omega_g\rangle \\
&=\frac{1}{(k!)^2(2\pi)^{2k}} \langle R^k\circ R^k\circ \omega_g , \frac{g^{2k}}{(2k)!}\rangle=\frac{1}{(k!)^2(2\pi)^{2k}(2k)!} \cc^{2k}\big(  R^k\circ R^k\circ \omega_g\big). 
\end{align*}
Where in the last step we used the fact that the contraction map $\cc$ is the adjoint of the  multiplication map by $g$, see \cite{Labbidoubleforms}.
\end{proof}
\begin{remark}
We remark that if $(e_i)$ is an orthonormal basis of $(V,g)$ and $R^k$ is seen as a bilinear form on $\Lambda^{2k}$  then
\[  \cc^{2k}\big(  R^k\circ R^k\circ \omega_g\big)=\sum_{i_1,...,i_{2k}=1}^{n}R^k\circ R^k\big(e_{i_1}\wedge ...\wedge e_{i_{2k}},\ast (e_{i_1}\wedge ...\wedge e_{i_{2k}})\big).\]
\end{remark}
For a Riemannian oriented manifold of dimension $n=4k$, let $R$ be its  Riemann curvature tensor seen as a $(2,2)$ double form.  The integral over $M$ of the function $p_k(R)$ is a \emph{topological invariant} called the $k$-th pontryagin number of $M$ and is denoted by $p_k(M)$.

\begin{proposition}\label{algebraic-thorpe}
Let $R$ be a $(2,2)$ double form on an oriented  Euclidean vector space $(V,g)$ of dimension $n=4k$. Then we have 
\begin{equation}
|{p}_k(R)| \leq \frac{1}{(k!)^2(2\pi)^{2k}} \| R^k\|^2.                       
\end{equation}
Furthermore, if equality occurs then the contraction $\cc R^k=0$.
\end{proposition}

\begin{proof}

The double form $R^q$ being symmetric then we have
\[{p}_k(R)=\frac{1}{(k!)^2(2\pi)^{2k}} \langle R^k\circ R^k,\omega_g\rangle=\frac{1}{(k!)^2(2\pi)^{2k}} \langle  R^k,  R^k\circ \omega_g\rangle.   \]
 Cauchy-Schwartz inequality shows that
\[  (k!)^2(2\pi)^{2k}   |{p}_k(R)|=| \langle  R^k,  R^k\circ \omega_g \rangle| \leq  \| R^k\|      \|  R^k\circ \omega_g \|.\] 
Next,  we apply  proposition \ref{mu-star} to show that $ \| R^k\circ \omega_g\| = \|    R^k\|$ as follows.
\[   \| R^k\circ \omega_g\|^2=\langle  R^k\circ \omega_g,  R^k\circ \omega_g\rangle=\langle R^k,R^k\rangle.\]
This completes the proof of the first part of the  proposition. To prove the second part, first we use   Lemma \ref{lemma-a} below  to show that 
\[c\big(R^k\circ \omega_g\big)=\mathfrak{S}R^k \circ \omega_g=0.\]
Finally, if equality holds, then  $R^k$ must be  proportional to $R^k\circ\omega_g$ and therefore $cR^k=c\big(R^k\circ \omega_g\big)=0$.
\end{proof}
\begin{lemma}\label{lemma-a}
Let $n=2p$ and $\omega$ be a $(p,p)$ double form. Then we have
\[\cc \big(\omega\circ\omega_g\big)=\mathfrak{S}\omega \circ \omega_g.\]
Where the last $\omega_g$ is the volume double form seen as a $(p-1,p+1)$ double form.

\end{lemma}
\begin{proof}
Let $(e_1,...,e_n)$ be an orthonormal basis of $V$. By linearity we may assume that $\omega =e_I^*\otimes e_J^*$, where $I,J$ two multi-indexes of length $p$. Then
\[ \omega\circ \omega_g=e_I^*\otimes \ast e_J^*.\]
Therefore,
\[\cc (\omega\circ \omega_g)=\sum_{i=1}^n i_{e_i}e_I^*\otimes i_{e_i}(\ast e_J^*).\]
A straightforward computation shows that $ i_{e_i}(\ast e_J^*)=\ast(e_i^*e_J^*)$, therefore
\[\cc (\omega\circ \omega_g)=\sum_{i=1}^n i_{e_i}e_I^*\otimes \ast(e_i^*e_J^*)=\sum_{i=1}^n( i_{e_i}e_I^*\otimes e_i^*e_J^*)\circ \omega_g=\mathfrak{S}\omega \circ \omega_g.\]

\end{proof}

\begin{remark}
If $W$ denotes the Weyl part of $R$ then in a similar way one can prove that 
\begin{equation*}
|{p}_k(R)| \leq \frac{1}{(k!)^2(2\pi)^{2k}} \| W^k\|^2.                       
\end{equation*}
\end{remark}

\section{Thorpe Manifolds}\label{Thorpe-man}
An oriented compact Riemannian manifold $(M,g)$ of dimension $n=4k$ is said to be a \emph{Thorpe manifold} if $\ast R^{k}=R^{k}.$ Where $R$ is the Riemann curvature tensor seen as a $(2,2)$ double form, $R^{2k}$ its exterior power and $\ast$ is the double Hodge star operator acting on double forms.\\
For $k=1$ we recover $4$-dimensional Einstein manifolds. For $n>4$, Thorpe manifolds are in general not Einstein ande vice versa, see for instance \cite{Kim}.\\
Recall that the Gauss-Bonnet theorem asserts that
\[ \int_Mh_{4k}\omega_g=(2\pi)^{2k}(2k)!\chi(M).\]

Where $h_{4k}$ is the  $4k$-Gauss-Bonnet curvature of $(M,g)$ which  is given by, see for instance \cite{Labbidoubleforms},
\[h_{4k}=\ast R^{2k}=\ast(R^kR^k)=\langle \ast R^k,R^k\rangle.\]
In particular, for a Thorpe manifold we have
\[h_{4k}=\|R^k\|^2\geq 0.\]
Therefore, the Euler-Poincar\'e characterestic of a compact Thorpe manifold of dimension $n=4k$ is always non-negative, and it is zero if and only if the manifold is $q$-flat. This remark can be refined as follows
\begin{theorem}\label{Thorpe-thm}
Let $(M,g)$ be a compact orientable $4k$-dimensional Thorpe manifold, then
\[\chi(M)\geq \frac{k!k!}{(2k)!}|p_k(M)|.\]
Where $p_k(M)$ is the $k$-th Pontrjagin number of $M$. Furthermore if equality holds then $M$ is $k$-Ricci flat, that is $\cc R^q=0.$
\end{theorem}
\begin{proof}
Using Proposition \ref{algebraic-thorpe}, we get
\begin{align*}
|p_k(M)|\leq  & \int_M|p_k(R)|\omega_g\\
& \leq \frac{1}{(k!)^2(2\pi)^{2k}}\int_M\|R^k\|^2\omega_g\\
&= \frac{1}{(k!)^2(2\pi)^{2k}}\int_M h_{4k}\omega_g\\
&=\frac{(2k)!}{k!k!}\chi(M).
\end{align*}
Furthermore, if equality holds then  $|p_k(R)|= \frac{1}{(k!)^2(2\pi)^{2k}} \|R^k\|^2$. Again by Proposition  \ref{algebraic-thorpe} the metric must then  be $k$-Ricci flat. This completes the proof of the theorem.

\end{proof}

We note that the first part of the Theorem was first proved by Thorpe \cite{Thorpe}. As a consequence of the second part of the previous  theorem, we get the following main  result of  \cite{Kim2}
\begin{corollary}\label{Kim}
A compact orientable manifold $M$  of dimension $8$ that is at the same time Einstein and Thorpe and satisfies $6\chi(M)=|p_2(M)|$  must be flat.
\end{corollary}
\begin{proof}
From one hand,  the previous Theorem \ref{Thorpe-thm}, once applied to the case $k=2$ shows that $\cc R^2=0$. Consequently, the second Gauss-Bonnet curvature  satisfies  $h_4=\frac{1}{4!}\cc^4 R^2=\frac{1}{4!}\cc^3(\cc R^2)=0$.\\
On the other hand,  an Einstein manifold with  identically zero $h_4$ must be  flat,  \cite{Labbidoubleforms}. This completes the proof of the corollary.\\
\end{proof}

Next, we are going to generalize the above corollary  to higher dimensions. First, we start by a definition
\begin{definition}[\cite{Labbi-eins}]
Let $0<2q<n$, we shall say that a  Riemannian $n$-manifold
 is \emph{ hyper $(2q)$-Einstein}  if the first contraction
 of the tensor  $R^q$
is proportional to the metric
 $g^{2q-1}$, that is
$$\cc R^q =\lambda g^{2q-1}.$$
\end{definition}
We recover the usual Einstein metrics for $q=1$. The following theorem generalizes a well known result about 4 dimensional Einstein manifolds
\begin{theorem}[\cite{Labbidoubleforms, Labbi-eins}]\label{hyper-eins}
Let  $k\geq 1$ and  $(M,g)$ be a hyper  $(2k)$-Einstein  manifold 
of dimension  $n\geq 4k$. Then  the Gauss-Bonnet curvature $h_{4k}$ of  $(M,g)$ is nonnegative. Furthermore, 
$h_{4k}\equiv 0$ if and only if  $(M,g)$ is $k$-flat.
\end{theorem}
Recall that $k$-flat means that the sectional curvature of  $R^k$ is identically zero.\\
Now, we are ready to state and prove a generalization of the above corollary \ref{Kim}
\begin{theorem}
A compact orientable manifold $M$  of dimension $8k$ that is at the same time hyper $2k$-Einstein and Thorpe and satisfies $\frac{(4k)!}{(2k)!(2k)!}\chi(M)=|p_{2k}(M)|$  must be $k$-flat.
\end{theorem}
We recover  corollary \ref{Kim} for $k=1$. 
\begin{proof}
 Theorem \ref{Thorpe-thm} shows that $\cc R^{2k}=0$. Consequently, the $4k$-th  Gauss-Bonnet curvature  satisfies  $h_{4k}=\frac{1}{(4k)!}\cc^{4k} R^{2k}=\frac{1}{(4k)!}\cc^{4k-1}(\cc R^{2k})=0$.Then
Theorem \ref{hyper-eins} shows that the manifold must be $k$-flat.
\end{proof}
\section{Final remarks}
\subsection{Mixed Pontrjagin numbers}
Let $(M,g)$ be a compact $4k$-dimensional Riemannian oriented manifold with Riemann curvature tensor $R$, seen as a $(2,2)$ double form,  and $ k_1, k_2, ..., k_m$ a  collection of natural numbers  such that $k_1+2k_2+...+mk_m =k$.\\
The Pontrjagin number $p_1^{k_1}p_2^{k_2}\cdots p_m^{k_m}$  of $M$ is defined by the integral over $M$ of the following $4k$-form
 $$ P_1^{k_1}P_2^{k_2}\cdots P_m^{k_m}:=\big(\underbrace{P_1\wedge \cdots \wedge P_1}_\text{$k_1$-times}\big)\wedge
\big(\underbrace{P_2\wedge \cdots \wedge P_2}_\text{$k_2$-times}\big)\wedge \cdots \wedge
\big(\underbrace{P_m\wedge \cdots \wedge P_m}_\text{$k_m$-times}\big).$$
Where for each $i$,  $P_i=P_{i}(R)$ denotes the $i$-th Pontrjagin form of $R$ as defined in section  \ref{Pontrjagin-forms}.\\
The above Pontrjagin numbers are important topological  invariants for a manifold, they
are oriented cobordism invariant, and together with Stiefel-Whitney numbers they determine an oriented manifold's oriented cobordism class. Furthermore, the signature and the $\hat{A}$ genus can be expressed explicitely through linear combinations of the above  Pontrjagin numbers.\\
We are now going to prove a Stehney type formula for all these numbers.
\begin{theorem}
With the above notations we have
\begin{equation*}
 P_1^{k_1}P_2^{k_2}\cdots P_m^{k_m}=\frac{(4k)!}{[(2k)!]^2(2\pi)^{2k}}  \Big(\prod_{i=1}^{m}\frac{[(2i)!]^2}{(i!)^{2k_i}(4i)!}\Big){\rm Alt}\Big[(R\circ R)^{k_1}(R^2\circ R^2)^{k_2}\cdots (R^m\circ R^m)^{k_m}\Big].
\end{equation*}
Where all the powers over double forms are taken with respect to the exterior product of double forms. In particular, the pontrjagin numbers are given by the integral
\begin{equation*}
\begin{split}
 &p_1^{k_1}P_2^{k_2}\cdots P_m^{k_m}=\\
&\frac{(4k)!}{[(2k)!]^2(2\pi)^{2k}}  \Big(\prod_{i=1}^{m}\frac{[(2i)!]^2}{(i!)^{2k_i}(4i)!}\Big)\int_M \frac{\cc^{2k}}{(2k)!}\Big[  \Big( (R\circ R)^{k_1}(R^2\circ R^2)^{k_2}\cdots (R^m\circ R^m)^{k_m}\Big)\circ \omega_g\Big]\omega_g.
\end{split}
\end{equation*}
Where $\cc$ is the contraction map of double forms and $\omega_g$ is the volume form as above.
\end{theorem}
\begin{proof}
Let $\omega_i$, for $i=1,2,\cdots,k$, be a collection of arbitrary  $(2p_i,2p_i)$ double forms, then  successive applications of Proposition \ref{Alt-is-endom} show that
\begin{equation}\label{first}
\frac{\big(\sum_{i=1}^{k}4p_i\big)!}{\big[\big(\sum_{i=1}^{k}2p_i\big)!\big]^2}{\rm Alt}\big(\omega_1\omega_2\cdots \omega_k\big)=\Big(\prod_{i=1}^{k}\frac{(4p_i)!}{\big[(2p_i)!\big]^2}\Big){\rm Alt}(\omega_1)\wedge {\rm Alt}(\omega_2)\wedge \cdots \wedge {\rm Alt}(\omega_k).
\end{equation}
In particular, if $\omega$ is an arbitrary $(2p,2p)$ double form one has
\begin{equation}\label{second}
\frac{    (4pk)!}{[(2pk)!]^2}{\rm Alt}(\omega^k)=\Big[ \frac{(4p)!}{[(2p)!]^2}\Big]^k   \underbrace{ {\rm Alt}(\omega)\wedge  \cdots \wedge {\rm Alt}(\omega)}_\text{$k$-times}.
\end{equation}
Using the previous two formulas, one can directly and without difficulties  complete the proof of the theorem. The second part can be proved easily by imitating the proof of proposition \ref{prop2.5}.
\end{proof}

\subsection{Stehney's formula for Pontrjagin forms}
For the seek of completeness,  we include here the derivation of   Stehney's formula (\ref{Stehney-formula})   from Chern's theorem \cite{Chern}.\\
Let $(M,g)$ be a Riemannian $n$-manifold with Riemann curvature tensor $R$,  $m\in M$ and  $(e_i)$   an orthonormal basis of the tangent space at $m$.  Chern theorem \cite{Chern} shows that at $m$  we have
$$P_k(R)=\frac{[(2k)!]^2}{(4k)!(2\pi)^{2k}(2^kk!)^2}\sum_{I}^{}\Omega_I\wedge\Omega_I.$$
Where the sum runs over all multi-indices $I=(i_1,...,i_{2k})$ such that $1\leq i_1<...<i_{2k}\leq n$, and 
$$\Omega_I(v_1,...,v_{2k})=2^k  R^k(v_1\wedge ...\wedge v_{2k};e_I).$$
We used the notation $e_I=e_{i_1}\wedge ...\wedge e_{i_{2k}}$. We emphasize that our convention for the wedge product is as defined by formula (\ref{Alt-conv}).
Next, let $J,K$ in the sums below run over all multi-indices of length $2k$ as it was the case for $I$ in the previous sum, then  we have
\begin{align*}
\frac{1}{2^{2k}}\sum_{I}^{}\Omega_I\wedge\Omega_I=& \frac{1}{2^{2k}}\sum_{I,J,K}^{}\Omega_I(e_J)e_J^*\wedge\Omega_I(e_K)e_K^*\\
=& \sum_{I,J,K}^{}R^k(e_J,e_I)R^k(e_K,e_I)e_J^*\wedge e_K^*\\
=& \sum_{J,K}^{}R^k\circ R^k(e_J,e_K) e_J^*\wedge e_K^*\\
=&\frac{(4k)!}{(2k)!(2k)!} {\rm Alt}\Big( \sum_{J,K}^{}R^k\circ R^k(e_J,e_K) e_J^*\otimes e_K^*\Big)\\
=& \frac{(4k)!}{(2k)!(2k)!} {\rm Alt}(R^k\circ R^k).
\end{align*}
Finally,
$$P_k(R)=\frac{[(2k)!]^2}{(4k)!(2\pi)^{2k}(2^kk!)^2}\sum_{I}^{}\Omega_I\wedge\Omega_I= \frac{2^{2k}{\rm Alt}(R^k\circ R^k)}{(2\pi)^{2k}(2^kk!)^2}= \frac{1}{(2\pi)^{2k}(k!)^2}{\rm Alt}(R^k\circ R^k) .$$

\address{Mathematics Department\\
College of Science\\
University of Bahrain\\
32038 Bahrain.\\
E-mail: mlabbi@uob.edu.bh}
\end{document}